\documentclass{hjm1}
\usepackage{verbatim}
\usepackage{amsthm}
\usepackage{amsmath}
\usepackage{amsfonts}
\usepackage{amssymb}
\usepackage[all,cmtip]{xy}
\newcommand{\bb}[1]{\mathbb{#1}}

\newcounter{egcounter}
\begin{document}
\title[Extending De Branges and Beurling]{A Class of Sub-Hardy Hilbert Spaces Associated with Weighted Shifts}

\author{Sneh Lata}
\address{Department of Mathematics,\\
         Shiv Nadar University,\\
         School of Natural Sciences,\\
         Gautam Budh Nagar - 203207,\\
         Uttar Pradesh, India}
\email{sneh.lata@snu.edu.in}

\author{ Dinesh Singh}
\address{Department of Mathematics,\\
        University of Delhi,\\
        Delhi 110007, India}
\email{dineshsingh1@gmail.com}


\begin{abstract} In this note we study sub-Hardy Hilbert spaces on which the the 
action of the operator of multiplication by the coordinate function $z$ is 
assumed to be weaker than that of an isometry. We identify such operators with a 
class of weighted shifts.  
The well known results of de Branges and Beurling
are deduced as corollaries .
\end{abstract}

\maketitle

\section{Introduction} 
There is a well known result of de Branges, see \cite{Sar2}, that characterizes all Hilbert
spaces which are contractively contained in the Hardy Space $H^2$ and on which the
operator of multiplication by the coordinate function $z$ acts isometrically. This
result is the starting point of a model theory for operators initiated by de Branges
and Rovnyak [4] and developed further in significant ways by de Branges \cite{deB}, 
Sarason \cite{Sar1} and by many others as mentioned in the references of Nikolskii \cite{Nik} 
and \cite{Sar1}. The starting point of this theory is the above referred theorem of de Branges which
in turn is a generalization of the famous invariant subspace theorem of Beurling \cite{Beu}. 
De Branges has actually generalized not only Beurling's theorem but also its
vector-valued generalizations due to Lax (finite multiplicity) and Halmos (infinite
multiplicity). In the scalar case, Sarason \cite{Sar2} has provided a simple proof of the
theorem of de Branges. It was further observed by Singh and Singh \cite{SS2} that no
continuity condition (in particular the contractivity condition) needs to be assumed
on the sub-Hardy Hilbert space contained in $H^2$ to classify the Hilbert spaces that are vector
subspaces of $H^2$ and on which $T_z,$ the operator of multiplication by the coordinate function $z,$ 
has an isometric action. Later, Agrawal and Singh
\cite{AS} and Singh \cite{Singh} produced theorems of de Branges type in the context of several
Banach spaces of analytic functions and in the context of two variables respectively.
De Branges' result was further generalized to the context of the operator of multiplication
by a finite Blaschke factor acting isometrically on a Hilbert space sitting
inside the scalar Hardy space $H^2$ by Singh and Thukral \cite{ST} thereby generalizing
the Lax theorem without taking recourse to vectorial function theory that is embedding 
in a scalar setting leading to a general inner-outer
factorization theory. Redett \cite{Red} extended the theorem of Singh 
\cite{Singh}, to the context of the $H^p$ spaces. In fact, Redett 
has extended the related results along the above discussed lines in 
several directions such as in \cite{Red1} and \cite{Red2}.  

In this paper we look at sub-Hardy Hilbert spaces on which the operator $T_z$ satisfies conditions more general than being an isometry and we identify such operators with a class of weighted shifts. Our theorem 
has three distinct and interesting features:
\begin{itemize}
\item[(i)] We impose simple conditions on the operator $T_z$ - of multiplication by the coordinate function $z.$
\item[(ii)] The characterization due to Singh and Singh \cite{SS2} follows as a corollary to our theorem thereby the characterizations due to de Branges \cite{deB} and Beurling \cite{Beu} also follow from our result.
\item[(iii)] Our conditions - as demonstrated in Remark \ref{ShimCond} - put the operator 
$T_z$ beyond the purview of Shimorin's theorem \cite{Shim}. 
\end{itemize} 

\section{Definitions and Preliminaries} The Hardy space $H^2$ is the 
Hilbert space of all analytic functions on the open unit disk $\bb D$ 
whose Taylor coefficients are absolutely square summable. The class of
all bounded analytic functions on $\bb D$, denoted by $H^{\infty}$, is a 
Banach algebra under the norm
$$
||f||_{\infty}= {\rm sup}\{ |f(z)| : z \in \bb D\}. 
$$
It is well known that $H^{\infty}$ is properly contained in $H^2$ and that an 
$f$ in $H^2$ also belongs to $H^{\infty}$ if and only if $f$ multiplies 
$H^2$ back into $H^2$. For a detailed account on Hardy spaces we refer to \cite{Dur} 
and \cite{Hof}.

Let $T$ be a bounded operator on a Hilbert space $H.$ We call $T$ a {\em weighted 
shift operator} with weight sequence $\{w_n\}_{n=0}^{\infty}$ if 
$Te_n = w_ne_{n+1}$, where $\{e_n\}$ is some orthonormal
basis for H and where $w_n$ is a bounded sequence of positive numbers. 
An operator $T$ is said to shift an orthogonal basis $\{h_n\}$ of $H$ if $Th_n = h_{n+1}$ for each $n$. Certain
elementary properties related to weighted shifts can be found in the 
survey article by Shields \cite{Shi}. We quote two very basic and 
standard facts from \cite{Shi} in this
section. The first theorem gives a necessary and sufficient condition 
for an operator to be a weighted shift.

\begin{thm}{\rm({\cite{Shi}, Proposition 6})}\label{Shi1} $T$ is an injective weighted 
shift operator on a Hilbert space $H$ if and only if $T$ shifts an 
orthogonal basis for $H.$  
\end{thm}

\begin{rmk} It turns out that if $T$ shifts an orthogonal basis 
$\{f_n\}$, then the basis $e_n$ and the weight sequence $\{w_n\}$ for 
which $T(e_n)=w_ne_{n+1}$ holds are given by 
$$
e_n=\frac{f_n}{||f_n||}, \ w_n=\frac{||f_{n+1}||}{||f_n||} \ for \ all \ n.
$$
\end{rmk}

The second theorem tells us that a weighted shift operator can be looked upon as the operator of multiplication by $z$ on the Hilbert space 
$$
H^2(\beta)=\left\{f(z)=\sum_{n=0}^{\infty}{\alpha}_nz^n : \sum_{n=0}^{\infty}|\alpha_n|^2\beta_n^2<\infty\right\}
$$
with the inner product 
$$
\langle{f,g}\rangle = \sum_{n=0}^{\infty}\alpha_n\overline{\gamma_n}\beta_n^2
$$
for all $f=\sum_{n=0}^{\infty}\alpha_nz^n$ and $g=\sum_{n=0}^{\infty}\gamma_nz^n$ in $H^2(\beta).$ Here $\{\beta_n\}$ is a sequence of positive numbers.

\begin{thm}{\rm(\cite{Shi}, Proposition 7)}\label{Shi2} The linear operator $T_z$ 
(multiplication by $z$) on $H^2(\beta)$ is unitarily equivalent to an 
injective weighted shift operator (with weight sequence $\{w_n\}$). 
Conversely, every injective weighted shift operator is unitarily 
equivalent to $T_z$ acting on $H^2(\beta)$ for some sequence $\beta.$
The precise relation between $\{w_n\}$ and $\beta$ is given by 
\begin{eqnarray*}
\beta_0 &=& 1\\
\beta_n &=& w_0\cdots w_{n-1} \ for \ all \ n\ge 1.
\end{eqnarray*}
\end{thm}

\section{Main Result}
\noindent The following is our main theorem.   

\begin{thm}\label{MT}
Let $M$ be a non-trivial Hilbert space contained in the Hardy space $H^2.$ 
Suppose the operator $T_z,$ which denotes multiplication by $z,$ is well defined on $M,$ 
and satisfies:
\begin{enumerate}
\item[(i)] There exists a $\delta>0$ such that $\delta ||f||_M \le ||T_zf||_M \le ||f||_M$ for all $f \in M.$
\item[(ii)] For each $n\in \bb N,$ $T_z^{*n}T_z^{n+1}(M)\subseteq T_z(M)$ (the adjoint of $T_z$ is with respect to the inner product on $M$).  
\end{enumerate} 
Then $T_z$ acts as a weighted shift on $M,$ and there exists a $b\in H^{\infty}$ such that 
$$
M=\overline{bH^2} \quad (the \ closure \ is \ in \ the \ norm \ of \ M)
$$ and 
$$
||bf||_M\le ||f||_2 \ \ for \ all \ f\in H^2.
$$ 
\end{thm}

\begin{proof} We first note that $T_z(M)$ is a closed subspace of $M$ as the bounded operator $T_z$ on $M$, using assumption (i), is bounded below. 

Put $N=M\ominus T_z(M).$ Note that $N$ is non-trivial. For if $N=\{0\}$, we have $M=T_z(M).$ 
This implies that for any $n\ge 1,$ an arbitrary $f\in M,$ can be written as $f=z^nf_n$ for some 
$f_n\in M.$ Since $f$ is holomorphic on $\bb D,$ this is possible only when $f=0,$ which contradicts 
the fact that $M$ is non-trivial. 

\noindent In view of assumption (ii), we can write 
\begin{eqnarray*}
M&=& N\oplus T_z(M)\\
&=& N\oplus T_z(N\oplus T_z(M))\\
&=& N\oplus T_z(N)\oplus T_z^2(M).
\end{eqnarray*}
Proceeding inductively we have, for every $n\in \bb N,$

\begin{equation}\label{odecomp}
M=N\oplus T_z(N)\oplus\cdots\oplus T_z^n(N)\oplus T_z^{n+1}(M)
\end{equation}

Choose $f=\sum_{n=0}^{\infty}\alpha_nz^n\in N$ with $||f||_M=1.$ We show that $f$ 
multiplies $H^2$ into $M.$ For any $g=\sum_{n=0}^{\infty}\gamma_nz^n\in H^2,$ write 
$g_n(z)=\sum_{k=0}^n\gamma_kz^k,$ and so $g_n\to g$ in $H^2.$  
For $n\in \bb N,$  
\begin{eqnarray}
||fg_n||^2_M &=& ||\sum_{k=0}^n\gamma_kz^kf||_M^2\nonumber \\
&=& \sum_{k=0}^n|\gamma_k|^2||z^kf||_M^2\nonumber \\
&\le & \sum_{k=0}^n|\gamma_k|^2\nonumber\\
&=&||g_n||_2^2\label{M2}.
\end{eqnarray}
Moreover, for $n\ge m\ge 1,$ by similar arguments as above, we get 
\begin{equation}\label{M3}
||fg_n-fg_m||^2_M \le ||g_n-g_m||_2^2.
\end{equation}
So, $\{fg_n\}$ is a Cauchy sequence in $M.$ Suppose $fg_n\to h\in M.$ 
Write 
$$
h=\gamma_0f+\gamma_1zf+\cdots+\gamma_kz^kf+z^{k+1}h_1
$$
for some $h_1\in H^2.$  
It can be seen that the coefficient of $z^k$ in the above equation is 
$$
\gamma_0\alpha_k+\gamma_1\alpha_{k-1}\cdots+\gamma_k\alpha_0
$$ 
which is same as the coefficient of $z^k$ in the formal product of the power series 
of $f$ and $g.$ Therefore $fg=h\in M,$ and with this we conclude that $fH^2\subseteq M\subseteq H^2.$ Hence $f\in H^{\infty}.$ The rest of our proof will be divided into the following two cases.

\noindent{\bf Case I.} Suppose there exists a function in $M$ that is non-zero at $0.$ Then, in view of Equation (\ref{odecomp}), we get a function in $N$ that is non-zero at $0.$ We assert that the dimension of $N$ is equal to 1. 

Suppose $f, g\in N$ such that $f\perp g$ in $M.$ We shall first show that  $fg\perp f$ in $M.$ If $f_n$ and $g_n$ stand for the $n^{th}$ partial sums of 
the Taylor series for $f$ and $g$ respectively, then it follows, using  Equations (\ref{odecomp}) and (\ref{M3}), that 
$f_ng \to fg$ and each $f_ng$ is orthogonal to $f$ in $M.$ Therefore, $fg\perp f$ in $M.$

Let $g= \gamma_0+zh$ for some $\gamma_0\in \bb C$ and $h\in H^2.$ Note that $f\in N.$ Therefore, $fg, \ fh$ are both in $M,$ and     
\begin{equation}\label{MT2}
fg=\gamma_0f+zfh
\end{equation}

Now since $zfh=T_z(fh)\in T_z(M)$, $zfh$ is orthogonal to $f.$ Then by taking 
inner product of both sides in (\ref{MT2}) with $f$ we obtain 
$$
\langle{fg,f}\rangle_M =\langle{\gamma_0f,f}\rangle_M
$$
But, by our earlier argument, $fg$ is orthogonal to $f$ in $M.$ Therefore  $\gamma_0 ||f||^2_M=0.$ Now if $f$ is non-zero, then $g(0)=\gamma_0 =0.$ 
In addition, if we assume $g\in N$ to be non-zero as well, then by using arguments similar to the one employed above, with roles of $f$ and $g$ being interchanged, we get $f(0)=0.$

Thus, if $dim(N)>1,$ every function in $N$ must be zero at $0$ which is not possible as  we do have a function in $N$ that is non-zero at $0$. Hence, $dim(N)=1.$  
 
Let $b\in N$ with $||b||_M=1.$ We shall now show that $bH^2$ is dense in $M.$ Since $b\in N, \ bH^2 \subseteq M$ which implies that 
$\overline{bH^2}\subseteq M.$ 
To establish the other containment, let $h\in M$ such that $h\perp bH^2.$ 
Then $h\perp bz^n$ for all $n\ge 0,$ and so $h\perp T_z^n(N)$  for all 
$n\ge 0$ as $N=\langle {b}\rangle.$ Therefore, using Equation (\ref{odecomp}), 
we get $h\in T_z^n(M)$ for all $n\ge 1.$ Thus, for each $n\ge 1,$ $h=z^nh_n$ for some $h_n\in M.$ Since $h$ is a holomorphic function on $\bb D,$ this is possible only when $h=0.$ Hence, $M= \overline{bH^2}.$ 
The fact that $||bf||_M\le ||f||_2$ for all $f\in H^2$ follows from Equation (\ref{M2}).

\vspace{1mm}
 
\noindent{\bf Case II.} Suppose $f(0)=0$ for all $f\in M.$ Let $n$ be the largest integer such that 
$f^{(n)}(0)=0$ for all $f\in M.$ Then each $f\in M$ is of the form 
$z^{n+1}g$ for some $g\in H^2,$ and also, 
there exists $h\in H^2$ such that $z^{n+1}h\in M$ and $h(0)\ne 0.$

Let $K=\{g\in H^2:z^{n+1}g\in M\}.$ Then K is a vector subspace of $H^2$ that is invariant under $T_z$ and contains an 
element that does not vanish at $0.$ Define $U:K\to M$ by $U(f)=z^{n+1}f.$ Clearly, $U$ is a vector space isomorphism from 
$K$ onto $M$ and therefore $||f||_K=||z^{n+1}f||_M$ defines a norm on K with 
respect to which $K$ becomes a Hilbert space and $U$ a unitary. 

For notational convenience, let $T_M$ and $T_K$ denote the operator $T_z$ acting on the Hilbert spaces $M$ 
and $K,$ respectively. Note that $T_M=UT_KU^*.$ Therefore, the Hilbert space $K$ satisfies all the hypotheses of the theorem. Also, K contains 
an element $f$ that is non-zero at 0. Hence, by Case I, there exists a $d\in H^{\infty}$ such that $K=\overline{dH^2}$ and 
$||df||_K\le ||f||_2$ for all $f\in H^2.$ Taking $b=z^{n+1}d$ we conclude that $M=\overline{bH^2}$ where $b\in H^{\infty}$ and $||bf||_M\le ||f||_2$ for all $f\in H^2.$   

In conclusion, for both the above cases, $\{bz^n\}$ is an orthogonal basis for $M=\overline{bH^2}.$ Since $T_z$ shifts this orthogonal basis, 
that is, $T_z(bz^n)=bz^{n+1},$ we conclude, using Theorem \ref{Shi1}, that $T_z$ acts as a weighted shift on $M.$ This completes the proof.
\end{proof}

\begin{rmk} In the above result we obtain that $bH^2$ is dense in $M.$ The subspace $bH^2$ is closed in $M$ 
if and only if there exists a $\delta>0$ such that 
\begin{equation}\label{Ine1}
\delta ||f||_M \le ||T_z^nf||_M \le ||f||_M 
\end{equation} 
for all $f \in M$ and all $n\ge 0.$ In this case, 
$\delta||f||_2\le ||bf||_M\le ||f||_2 $ for all $f\in H^2.$ Note that (\ref{Ine1})  is nothing but assumption (i) in the above 
theorem with $T_z^n$ in place of $T_z.$ 
\end{rmk}
 
To see that our above remark is valid we proceed as follows. If the inequalities in (\ref{Ine1}) hold true, then by using arguments similar to the ones employed to obtain (\ref{M2}) we get $\delta||h||_2\le ||bh||_M\le ||h||_2$ for all 
$h\in H^2.$ From this it easily follows that $bH^2$ is closed in $M.$ To prove the reverse implication, suppose $bH^2$ is a closed subspace of $M.$ Then the one-to-one contraction  $h\mapsto bh$ from $H^2$ onto $bH^2$, using Open Mapping Theorem, is actually invertible. Thus, there exists a positive constant $c$ such that $c||h||_2\le ||bh||_M\le ||h||_2$ for all $h\in H^2.$ This implies, for any $k, \ n\ge 0,$  
$c||bz^k||_M\le ||bz^{n+k}||_M\le ||bz^k||_M$ which yields the Inequalities in (\ref{Ine1}). 
 
A weighted Hardy space $H^2_\beta$ where $\{\beta_n\}$ is a decreasing 
sequence of positive numbers bounded away from zero serves as an easy example of $M$ for which (\ref{Ine1}) is satisfied. Here even though $M=H^2(\beta)=H^2$ as a vector space, but by choosing $\beta_n$ carefully we can guarantee that $T_z$ does not act as an isometry on $M.$ For example, take $\beta_n=(n+3)^{1/(n+3)}, \ n\ge 0.$

\begin{rmk}{\label{ShimCond}}
As our proof of the Theorem \ref{MT} indicates, we have obtained a decomposition of $M$ that allows us to proceed with the rest of the proof. There is a decomposition, due to Shimorin \cite{Shim}, of Hilbert spaces $H$ connected with operators $T$ that satisfy one of the conditions:
$$
||T^2x||^2 +||x||^2 \le 2 ||Tx||^2 \ \ {\rm for \ any} \ x\in H
$$
or
$$
||Tx+y||^2 \le 2\left(||x||^2 + ||Ty||^2\right) \ \ {\rm for \ any}\  x, y\in H
$$

This is not applicable in our situation as the following example indicates. Let $\beta_n=\frac{1}{2^{\frac{n}{2}}}$ 
when $n$ is even, and $\beta_n=\frac{1}{2^{\frac{n-1}{2}}}$ when $n$ is odd. Then $T_z,$ multiplication by $z,$ 
on the weighted Hardy space $H^2(\beta)$ satisfies the conditions of Theorem \ref{MT}, but fails the 
conditions of Shimorin's theorem. 
\end{rmk}

\section{Important Consequences of the Main Theorem}
In this section we discuss some important and immediate consequences of Theorem \ref{MT}.

\begin{cor}\label{cor1}{\rm \cite{SS2}} Let $M$ be a Hilbert space 
contained in $H^2$ as a vector subspace and such that $T_z(M)\subseteq M$ 
and let $T_z$ act isometrically on $M.$ Then there exists a $b\in H^{\infty}$ 
such that $M=bH^2,$ and $||bf||_M=||f||_2$ for all $f\in H^2.$  
\end{cor}
\begin{proof} The two conditions stated in Theorem (\ref{MT}) and the set 
of inequalities in (\ref{Ine1}) (with $\delta =1$) are trivially satisfied 
when $T_z$ is an isometry. So, $M=bH^2$ for some $b\in H^{\infty},$ and  $||bf||_M=||f||_2$ for all $f\in H^2.$ 
\end{proof}

This result of Singh and Singh generalizes the characterization result of de Branges \cite{deB} and therefore of Beurling \cite{Beu} as well. Hence, 
our theorem also implies these characterizations. 

\subsection*{Acknowledgements} Both authors
thank the Mathematical Sciences Foundation,
Delhi for support and facilities needed to complete
the present work.


\begin{thebibliography}{99}
\bibitem{AS} S. Agrawal and D. Singh, {\em de Branges spaces contained in some Banach spaces of analytic
functions}, Illinois J. Math. {\bf 39} (1995), no. 3, 351-357.
\bibitem{Beu} A. Beurling, {\em On two problems concerning linear transformations in Hilbert space}, 
Acta Math. {\bf 81} (1949), 239-255.
\bibitem{deB} L. de Branges, {\em Square summable power series}. (manuscript)
\bibitem{deBR} L. de Branges and J. Rovnyak, {\em Square summable power series}, Holt, Rhinehart and 
Winston, 1966.
\bibitem{Dur} P.L. Duren, {\em Theory of Hp spaces}, Academic Press, 1970.
\bibitem{Hof} K. Hoffman, {\em Banach Spaces of Analytic Functions}, Prentice Hall, 1962.
\bibitem{Nik} N. K. Nikolskii, {\em Operators, functions and systems: an easy reading, Mathematical Surveys
and Monographs}, {\bf 92}. American Mathematical Society, Providence, RI, 2002.
\bibitem{Red} D.A. Redett, {\em Brangesian Spaces in $H^p(\bb T^2)$}, Proc. Amer. Math. Soc. {\bf 133} (2005), 
2689-2695.
\bibitem{Red1} D. A. Redett, {\em "Beurling Type" Subspaces of $L^p(\bb T^2)$ and $H^p(\bb T^2)$}, Proc. Amer. Math. Soc. {\bf 133}
(2005), 1151-1156.
\bibitem{Red2} D. A. Redett, {\em Sub-Lebesgue Hilbert Spaces on the Unit Circle}, Bull. London Math. Soc. {\bf 37}
(2005), 793-800.
Acad. Ser. A Math. Sci. {\bf 87} (2011), 56-59.
\bibitem{Sar1} D. Sarason, {\em Sub-Hardy Hilbert Spaces in the Unit Disk}, Lecture Notes in the Math. Sciences {\bf 10}, 
Wiley, New York, 1994.
\bibitem{Sar2} D. Sarason, {\em Shift-invariant spaces from the Brangesian point of view}, Math. Surveys 
Monograph {\bf 21} (1986), Amer. Math. Soc., Providence, RI, 153-166.
\bibitem{Shi} A. Shields, {\em Weighted shift operators and analytic function theory}, Topics in operator theory,
Math. Surveys {\bf 13} (1974), Amer. Math. Soc., Providence, R.I., 49-128.
\bibitem{Shim} S. Shimorin, {\em Wold-type decomposition and wandering subspaces for operators close to isometries}, 
J. reine angew. Math. {\bf 531} (2001), 147-189.
\bibitem{Singh} D. Singh, {\em Brangesian spaces in the polydisk}, Proc. Amer. Math. Soc. {\bf 110} (1990), 971-977.
\bibitem{SS2} U.N. Singh and D. Singh, {\em On a theorem of de Branges}, Indian Journal of Math. {\bf 33} (1991),
1-5.
\bibitem{ST} D. Singh and V. Thukral, {\em Multiplication by finite Blaschke factors on de Branges spaces}, J.
Operator Theory {\bf 37} (1997), 223-245.
\end{thebibliography}
\end{document}